\documentclass[a4paper,12pt,intlimits,oneside]{amsart}
\usepackage{amsmath}
\usepackage{amsthm}
\usepackage{latexsym}
\usepackage{amssymb}
\usepackage{xcolor}
\usepackage{graphicx}
\numberwithin{figure}{section}
\def\R{{\mathbb R}}
\def\C{{\mathbb C}}
\def\T{{\mathbb T}}
\def\Z{{\mathbb Z}}

\def\s{\vskip 0.25cm\noindent}
\def\e{\varepsilon}
\def\pa{\partial}
\def\build#1_#2^#3{\mathrel{
\mathop{\kern 0pt#1}\limits_{#2}^{#3}}}
\def\td_#1,#2{\mathrel{\mathop{\build\longrightarrow_{#1\rightarrow #2}^{}}}}
\newtheorem{theorem}{Theorem}
\newtheorem{corollary}{Corollary}

\newtheorem{lemma}{Lemma}
\newtheorem{remark}{Remark}

\begin{document}
\title[Explicit formula for the cubic Szeg\H{o} equation]{An explicit formula \\ for the cubic Szeg\H{o} equation}
\author{Patrick G\'erard}
\address{Universit\'e Paris-Sud XI, Laboratoire de Math\'ematiques
d'Orsay, CNRS, UMR 8628, et Institut Universitaire de France} \email{{\tt Patrick.Gerard@math.u-psud.fr}}
\author[S. Grellier]{Sandrine Grellier}
\address{F\'ed\'eration Denis Poisson, MAPMO-UMR 6628,
D\'epartement de Math\'ematiques, Universit\'e d'Orleans, 45067
Orl\'eans Cedex 2, France} \email{{\tt
Sandrine.Grellier@univ-orleans.fr}}


\date{April 9, 2013}
\thanks{Part of this work was made while the authors were visiting CIRM in Luminy. They are grateful to this institution
for its warm hospitality. Moreover, this paper benefited from discussions with several colleagues, in particular T. Kappeler,
H. Koch, S. Kuksin and M. Zworski. We wish to thank them deeply. }

\begin{abstract}
We derive an explicit formula for the general solution of the cubic Szeg\H{o} equation and of the evolution equation of the corresponding hierarchy. As an application, we prove that all the solutions corresponding to finite rank Hankel operators are quasiperiodic.
\end{abstract}

\maketitle

\section{Introduction}
This paper is a continuation of the study of dynamical properties of an integrable system introduced by the authors in \cite{GG}, \cite{GG2}.
As an evolution equation, the cubic Szeg\H{o} equation is a simple model of non dispersive dynamics. More precisely, it can be identified as a first order Birkhoff normal form for a certain nonlinear wave equation, see \cite{GG3}.  As an Hamiltonian equation, it was proved in
\cite{GG} to admit a Lax pair  and finite dimensional invariant submanifolds corresponding to some finite rank conditions. In \cite{GG2},  action angle variables were introduced on generic subsets of the phase space, and on open dense subsets of the finite rank submanifolds. However, unlike the KdV equation or the one dimensional cubic nonlinear Schr\"odinger equation, this integrable system displays some degeneracy, since the collection of its conservation laws do not control the high regularity of the solution, as observed in \cite{GG}. An important consequence of this instability phenomenon is that the action angle variables cannot be extended to the whole phase space, even when restricted to one of the  finite rank submanifolds. Our purpose in this paper is to prove a formula for the general solution of the initial value problem for this equation. In the case of generic data, this formula reduces to the one given by the action angle variables above. However, the formula  enables to study  the non generic case too, and allows in particular to establish the quasiperiodicity of all solutions lying in one of the above finite rank submanifolds, despite the already mentioned lack of a global system of action--angle variables. Finally, this formula is also very useful to revisit the instability phenomenon displayed in \cite{GG}. We now introduce the general setting of this equation.
\subsection{The setting}
Let $\T =\R /2\pi \Z $, endowed with the Haar integral
$$\int _\T f :=\frac1{2\pi} \int _0^{2\pi }f(x)\, dx\ .$$
On $L^2(\T )$, we use the inner product
$$(f\vert g):=\int _\T f\overline g\ .$$
The family of functions $({\rm e}^{ikx})_{ k\in \Z} $ is an orthonormal basis of $L^2(\T )$,
on which the components of $f\in L^2(\T )$ are the Fourier coefficients
$$\hat f(k):=(f\vert {\rm e}^{ikx})\ .$$
We introduce the closed subspace
$$L^2_+(\T ):=\{ u\in L^2(\T ): \forall k<0, \hat u(k)=0\}\ . $$
Notice that elements $u\in L^2_+(\T )$ identify to traces of  holomorphic functions $\underline u$ on the unit disc $D$ such that
$$\sup _{r<1}\int _0^{2\pi }\vert \underline u(r{\rm e}^{ix})\vert ^2dx <\infty \ ,$$
via the correspondence
$$\underline u(z):=\sum _{k=0}^\infty \hat u(k)z^k\ ,\ z\in D\ ,\ u(x)=\lim _{r\rightarrow 1} \underline u(r{\rm e}^{ix})\ ,$$
which establishes a bijective isometry between $L^2_+(\T )$ and the Hardy space of the disc.
\s
We denote by $\Pi$ the orthogonal projector from $L^2(\T )$ onto $L^2_+(\T )$,  known as the Szeg\H{o} projector,
$$\Pi \left (\sum _{k=-\infty }^\infty \hat f(k){\rm e}^{ikx} \right )=\sum _{k=0 }^\infty \hat f(k){\rm e}^{ikx}\ .$$
On $L^2_+(\T )$, we introduce the symplectic form
$$\omega (h_1,h_2)={\rm Im}(h_1\vert h_2)\ .$$
The densely defined energy functional
$$E(u):=\frac 14\int _\T \vert u\vert ^4\ ,$$
 formally corresponds to the Hamiltonian evolution equation,
 \begin{equation}\label{szego}
 i\dot u=\Pi (\vert u\vert ^2u)\ ,
\end{equation}
which we called the cubic Szeg\H{o} equation. In \cite{GG}, we solved the initial value problem  for this equation on the intersections of Sobolev spaces with $L^2_+(\T )$. More precisely, define, for $s\ge 0$, 
$$H^s_+(\T ):=H^s(\T )\cap L^2_+(\T )=\{ u\in L^2_+(\T ): \sum _{k=0}^\infty \vert \hat u(k)\vert ^2(1+k^2)^s<\infty \}\ .$$
Then equation (\ref{szego}) defines  a smooth flow on  $H^s_+(\T )$  for $s>\frac 12$, and a continuous flow on $H^{\frac 12}_+(\T )$.
The main result of this paper provides an explicit formula for the solution of this initial value problem.
\subsection{Hankel operators and the explicit formula}
Let $u\in H^{\frac 12}_+(\T )$. We denote by $H_u$ the $\C $--antilinear operator defined on $L^2_+(\T )$ as
$$H_u(h)=\Pi (u\overline h)\ ,\ h\in L^2_+(\T )\ .$$
In terms of Fourier coefficients, this operator reads
$$\widehat{H_u(h)} (n)=\sum _{p=0}^\infty \hat u(n+p)\overline{\hat h(p)}\ .$$
In particular, its Hilbert--Schmidt norm is finite since $u\in H^{\frac 12}_+(\T )$. We call $H_u$ the Hankel operator of symbol $u$.
Notice that this definition is different from the standard ones used in references \cite{N}, \cite{P}, where Hankel operators are rather defined as linear operators from $L^2_+$ into its orthogonal complement. The link between these two definitions can be easily established by means of the involution
$$f^\sharp (x)= {\rm e}^{-ix} \overline {f(x)}\ .$$
Notice that, with our definition, $H_u$ satisfies the following self adjointness identity,
\begin{equation}\label{selfadjoint}
(H_u(h_1)\vert h_2)=(H_u(h_2)\vert h_1)\ ,\ h_1, h_2\in L^2_+(\T )\ .
\end{equation}
A fundamental property of Hankel operators is their connection with the shift operator $S$, defined on $L^2_+(\T )$ as
$$ Su (x)= {\rm e}^{ix}u(x)\ .$$
This property reads
$$S^*H_u=H_uS=H_{S^*u}\ ,$$
where $S^*$ denotes the adjoint of $S$. We denote by $K_u$ this operator, and call it the shifted Hankel operator of symbol $u$.
Notice that $K_u$ is Hilbert--Schmidt and self adjoint as well. As a consequence, operators $H_u^2$ and $K_u^2$ are $\C $--linear
trace class positive operators on $L^2_+(\T )$. Moreover, they are related by the following important identity,
\begin{equation}\label{K2H}
K_u^2=H_u^2-(\cdot \vert u)u\ .
\end{equation}
\begin{theorem}\label{explicit}
Let $u_0\in H^{\frac 12}_+(\T )$, and $u\in C(\R ,H^{\frac 12}_+(\T ))$ be the solution of equation (\ref{szego}) such that $u(0)=u_0$.
Then
$$\underline u(t,z)=((I-z{\rm e}^{-itH_{u_0}^2}{\rm e}^{itK_{u_0}^2}S^*)^{-1}{\rm e}^{-itH_{u_0}^2}u_0\, \vert \, 1)\ .$$
\end{theorem}
The proof of this theorem will be given in section \ref{preuve}. It is a non trivial consequence of the Lax pair structure recalled in section \ref{Lax}.
Our second result concerns the special case of data $u_0$ such that $H_{u_0}$ is of finite rank. In this case, operators $S^*, H_{u_0}^2, K_{u_0}^2$ act on a finite dimensional space containing $u_0$, and the implementation of the above formula reduces to diagonalization of matrices.
\subsection{Finite rank manifolds and quasiperiodicity}\label{finite}
Let $d$ be a positive integer. We denote by $\mathcal V(d)$ the set of $u\in H^{\frac 12}_+(\T )$ such that
$$ {\rm rk} H_{u}=\left [\frac{d+1}2\right ]\ ,\ {\rm rk} K_{u}= \left [\frac{d}2\right ]\ ,$$
where $[x]$ denotes the integer part of $x\in \R $. Using Kronecker's theorem \cite{Kr}, one can show that $\mathcal V(d)$ is a complex K\"ahler 
submanifold of $L^2_+(\T )$ of dimension $d$ --- see the appendix of \cite{GG} ---, consisting of rational functions of ${\rm e}^{ix}$.
More precisely,  $\mathcal V(d)$ consists of functions of the form
$$u(x)=\frac{A({\rm e}^{ix})}{B({\rm e}^{ix})}\ ,$$
where $A, B$ are polynomials with no common factors, $B$ has no zero in the closed unit disc, $B(0)=1$, and
\begin{itemize}
\item If $d=2N$ is even, the degree of $A$ is at most $N-1$ and the degree of $B$ is exactly $N$.
\item If $d=2N+1$ is odd, the degree of $A$ is exactly $N$ and the degree of $B$ is at most $N$.
\end{itemize}
Using the Lax pair structure recalled in section \ref{Lax}, $\mathcal V(d)$ is invariant through the flow of (\ref{szego}). 
\begin{theorem}\label{quasip}
For every $u_0\in \mathcal V(d)$, the map
$$t\in \R \mapsto u(t)\in \mathcal V(d)$$
is quasiperiodic. More precisely, there exist a positive integer $n$, real numbers $\omega _1,\cdots ,\omega _n$, and a smooth mapping
$$\Phi : \T ^n\rightarrow \mathcal V(d)$$
such that, for every $t\in \R $,
$$u(t)=\Phi (\omega _1t,\cdots ,\omega _nt)\ .$$
In particular, for every $s>\frac 12$,
\begin{equation}\label{Hs}
\sup _{t\in \R }\Vert u(t)\Vert _{H^s} <+\infty \ .
\end{equation}
\end{theorem}
Notice that property (\ref{Hs}) was established in Theorem 7.1 of \cite{GG} under the additional generic assumption that $u_0$ belongs to 
$\mathcal V(d)_{\rm gen}$, namely that  the vectors $H_{u_0}^{2n}(1), n=1,\dots , N=\left [\frac{d+1}2\right ]$,  are linearly independent. Our general formula allows us to extend property (\ref{Hs}) to all data in $\mathcal V(d)$. However, it should be emphasized that, while it is clear from the arguments of Lemma 5 in \cite{GG} that estimate (\ref{Hs}) is uniform if $u_0$  varies in a compact subset of $\mathcal V(d)_{\rm gen}$, (\ref{Hs}) does not follow from an a priori estimate on the whole of $\mathcal V(d)$, in the sense that one can find families of data $(u_0^\e )$ in $\mathcal V(d)$, belonging to a compact subset of $\mathcal V(d)$, in particular bounded in all $H^s$, and such that
$$\sup _\e \sup _{t\in \R }\Vert u^\e (t)\Vert _{H^s} =\infty \ ,\ s>\frac 12\ ,$$
see corollary 5 of \cite{GG}.
We shall revisit this phenomenon in section \ref{instab} thanks to the explicit formula of Theorem \ref{explicit}.
\s  Finally, let us mention that the generalization of property (\ref{Hs}) to non finite rank solutions is an open problem.
\subsection{Organization of the paper}
Section \ref{Lax} is devoted to recalling the crucial Lax pair structure attached to equation (\ref{szego}). As a fundamental consequence,
$H_{u(t)}$ and $K_{u(t)}$ remain unitarily equivalent to their respective initial data. In section \ref{preuve}, we take advantage of this structure 
to derive Theorem \ref{explicit}. In section \ref{instab}, we apply this theorem to the particular case of data $u_0$ belonging to $\mathcal V(3)$,
which sheds a new light on the instability phenomenon. The next two sections are devoted to the proof of Theorem \ref{quasip}. As a preparation, we first  generalize the explicit formula to Hamiltonian flows associated to energies 
$$J^y(u):=((I+yH_u^2)^{-1}(1)\vert 1)\ ,$$
where $y$ is a positive parameter. The quasi periodicity theorem then follows by observing, through an interpolation argument, that the map $\Phi $ in the statement of Theorem \ref{quasip} can be defined as the value at time $1$ of the Hamiltonian flow corresponding to  a suitable  linear combination of energies $J^y$. 
\section{The Lax pair structure}\label{Lax}
In this section, we recall the Lax pairs associated to the cubic Szeg\H{o} equation, see \cite{GG}, \cite{GG2}. 
First we introduce the notion of a Toeplitz operator. Given $b\in L^\infty (\T )$, we define $T_b:L^2_+\to L^2_+$ as
$$T_b(h)=\Pi (bh)\ ,\ h\in L^2_+\ .$$
Notice that $T_b$ is bounded and $T_b^*=T_{\overline b}$. 
The starting point is the following lemma.
\begin{lemma}\label{abc}
Let $a, b, c\in H^s_+$, $s>\frac 12$. Then
$$H_{\Pi (a\overline b c)}=T_{a\overline b}H_c+H_aT_{b\overline c} -H_aH_bH_c\ .$$
\end{lemma}
\begin{proof}
Given $h\in L^2_+$, we have
\begin{eqnarray*}
H_{\Pi (a\overline b c)}(h)&=&\Pi (a\overline b c\overline h)=\Pi (a\overline b\Pi (c\overline h))+\Pi (a\overline b(I-\Pi )(c\overline h))\\
&=&T_{a\overline b}H_c(h)+H_a(g)\ ,\ g:=b\overline {(I-\Pi )(c\overline h)}\ .
\end{eqnarray*}
Since $g\in L^2_+$, 
$$g=\Pi (g)=\Pi (b\overline ch)-\Pi (b\overline {\Pi (c\overline h)})=T_{b\overline c}(h)-H_bH_c(h)\ .$$
This completes the proof.
\end{proof}
Using Lemma \ref{abc} with $a=b=c=u$, we get
\begin{equation}\label{Hu3}
H_{\Pi (\vert u\vert ^2u)}=T_{\vert u\vert ^2}H_u+H_uT_{\vert u\vert ^2}-H_u^3\ .
\end{equation}
\begin{theorem}\label{Lax pair}
Let $u\in C^\infty (\R ,H^s_+), s>\frac 12, $ be a solution of (\ref{szego}). Then
\begin{eqnarray*}
\frac{dH_u}{dt}&=&[B_u,H_u]\ ,\ B_u:=\frac i2H_u^2-iT_{\vert u\vert ^2}\ ,\\
\frac{dK_u}{dt}&=&[C_u,K_u]\ ,\ C_u:=\frac i2K_u^2-iT_{\vert u\vert ^2}\ .
\end{eqnarray*}
\end{theorem}
\begin{proof} Using equation (\ref{szego}) and identity (\ref{Hu3}), 
$$\frac{dH_u}{dt}=H_{-i\Pi (\vert u\vert ^2u)}=-i H_{\Pi (\vert u\vert ^2u)}=-i(T_{\vert u\vert ^2}H_u+H_uT_{\vert u\vert ^2}-H_u^3)\ .$$
Using the antilinearity of $H_u$, this leads to the first identity. For the second one, we observe that
\begin{equation}\label{KPi}
K_{\Pi (\vert u\vert ^2u)}=H_{\Pi (\vert u\vert ^2u)}S=T_{\vert u\vert ^2}H_uS+H_uT_{\vert u\vert ^2}S-H_u^3S\ .
\end{equation}
Moreover, notice that
$$T_b(Sh)=ST_b(h)+(bSh\vert 1)\ .$$
In the case $b=\vert u\vert ^2$, this gives
$$T_{\vert u\vert ^2}Sh =ST_{\vert u\vert ^2}h+(\vert u\vert ^2Sh\vert 1)\ .$$
Moreover,
$$(\vert u\vert ^2Sh\vert 1)=(u\vert u\overline {Sh})= (u\vert K_u(h))\ .$$
Consequently,
$$H_uT_{\vert u\vert ^2}Sh=K_uT_{\vert u\vert ^2}h+(K_u(h)\vert u)u\ .$$
Coming back to (\ref{KPi}), we obtain
$$K_{\Pi (\vert u\vert ^2u)}=T_{\vert u\vert ^2}K_u+K_uT_{\vert u\vert ^2}-(H_u^2-(\cdot \vert u)u) K_u\ .$$
Using identity (\ref{K2H}), this leads to
\begin{equation}\label{Ku3}
K_{\Pi (\vert u\vert ^2u)}=T_{\vert u\vert ^2}K_u+K_uT_{\vert u\vert ^2} - K_u^3\ .
\end{equation}
The second identity is therefore a consequence of antilinearity and of 
$$\frac{dK_u}{dt}=-iK_{\Pi (\vert u\vert ^2u)}\ .$$
\end{proof}
Observing that $B_u, C_u$ are linear and antiselfadjoint, we obtain, following a classical argument due to Lax \cite{L},
\begin{corollary}\label{UV}
Under the conditions of Theorem \ref{Lax pair}, define $U=U(t)$, $V=V(t)$ the solutions of the following linear ODEs
on $\mathcal L(L^2_+)$,
$$\frac {dU}{dt}=B_uU\ ,\ \frac{dV}{dt}=C_uV\ ,\ U(0)=V(0)=I\ .$$
Then $U(t), V(t)$ are unitary operators and
$$H_{u(t)}=U(t)H_{u(0)} U(t)^*\ ,\ K_{u(t)}=V(t)K_{u(0)} V(t)^*\ .$$
\end{corollary}
\section{Proof of the formula}\label{preuve}
In this section, we prove Theorem \ref{explicit}. Our starting point is the following identity, valid for every $v\in L^2_+$,
\begin{equation}\label{Taylor}
\underline v(z)=((I-zS^*)^{-1}v\vert 1)\ ,\ z\in D\ .
\end{equation}
Indeed, the Taylor coefficient of order $n$ of the right hand side at $z=0$ is 
$$((S^*)^nv\vert 1)=(v\vert S^n1)=\hat v(n)\ ,$$
which coincides with the Taylor coefficient of order $n$ of the left hand side. Let $u\in C^\infty (\R ,H^s_+)$ be a solution of (\ref{szego}), $s>\frac 12$.  Applying (\ref{Taylor}) to $v=u(t)$ and using the unitarity of $U(t)$, we get
$$\underline u(t,z)=((I-zS^*)^{-1}u(t)\vert 1)=(U(t)^*(I-zS^*)^{-1}u(t)\vert U(t)^*1)\ ,$$
which yields
\begin{equation}\label{ubar}
\underline u(t,z)=((I-zU(t)^*S^*U(t))^{-1}U(t)^*u(t)\vert U(t)^*1)\ .
\end{equation}
We shall identify successively $U(t)^*1, U(t)^*u(t),$ and the restriction of  $U(t)^*S^*U(t)$ on the  range of $H_{u_0}$. We begin with $U(t)^*1$,
$$\frac{d}{dt}U(t)^*1=-U(t)^*B_u(1)\ ,$$
and
$$B_u(1)=\frac i2 H_u^2(1)-iT_{\vert u\vert ^2}(1)=-\frac i2H_u^2(1)\ .$$
Hence
$$\frac{d}{dt}U(t)^*1=\frac i2U(t)^*H_u^2(1)=\frac i2H_{u_0}^2U(t)^*1\ ,$$
where we have used corollary \ref{UV}. This yields
\begin{equation}\label{U*1}
U(t)^*1={\rm e}^{i\frac t2 H_{u_0}^2}(1)\ .
\end{equation}
Consequently,
$$U(t)^*u(t)=U(t)^*H_{u(t)}(1)=H_{u_0}U(t)^*(1)=H_{u_0}{\rm e}^{i\frac t2 H_{u_0}^2}(1)\ ,$$
and therefore
\begin{equation}\label{U*u}
U(t)^*u(t)={\rm e}^{-i\frac t2 H_{u_0}^2}(u_0)\ .
\end{equation}
Finally,
$$U(t)^*S^*U(t)H_{u_0}= U(t)^*S^*H_{u(t)}U(t)=U(t)^*K_{u(t)}U(t)\ ,$$
and therefore
\begin{equation}\label{U*S*UH}
U(t)^*S^*U(t)H_{u_0}=U(t)^*V(t)K_{u_0}V(t)^*U(t)\ .
\end{equation}
On the other hand,
\begin{eqnarray*}
\frac d{dt}U(t)^*V(t)&=& -U(t)^*B_{u(t)} V(t)+U(t)^*C_{u(t)}V(t)=U(t)^*(C_{u(t)}-B_{u(t)})V(t)\\
&=&\frac i2 U(t)^*(K_{u(t)}^2-H_{u(t)}^2)V(t) =\frac i2 (U(t)^*V(t)K_{u_0}^2-H_{u_0}^2U(t)^*V(t))\ .
\end{eqnarray*}
We infer
$$U(t)^*V(t)={\rm e}^{-i\frac t2 H_{u_0}^2}{\rm e}^{i\frac t2 K_{u_0}^2}\ .$$
Plugging this identity into (\ref{U*S*UH}), we obtain
\begin{eqnarray*}
U(t)^*S^*U(t)H_{u_0}&=&{\rm e}^{-i\frac t2 H_{u_0}^2}{\rm e}^{i\frac t2 K_{u_0}^2}K_{u_0} {\rm e}^{-i\frac t2 K_{u_0}^2}{\rm e}^{i\frac t2 H_{u_0}^2}\\
&=& {\rm e}^{-i\frac t2 H_{u_0}^2}{\rm e}^{i t K_{u_0}^2}K_{u_0} {\rm e}^{i\frac t2 H_{u_0}^2}\\
&=&{\rm e}^{-i\frac t2 H_{u_0}^2}{\rm e}^{i t K_{u_0}^2}S^*H_{u_0} {\rm e}^{i\frac t2 H_{u_0}^2}\\
&=&{\rm e}^{-i\frac t2 H_{u_0}^2}{\rm e}^{i t K_{u_0}^2}S^*{\rm e}^{-i\frac t2 H_{u_0}^2}H_{u_0} \ .
\end{eqnarray*}
We conclude that, on the range of $H_{u_0}$, 
\begin{equation}\label{U*S*U}
U(t)^*S^*U(t)={\rm e}^{-i\frac t2 H_{u_0}^2}{\rm e}^{i t K_{u_0}^2}S^*{\rm e}^{-i\frac t2 H_{u_0}^2}\ .
\end{equation}
It remains to plug identities (\ref{U*1}), (\ref{U*u}), (\ref{U*S*U}) into (\ref{ubar}). We finally obtain
\begin{eqnarray*}
\underline u(t,z)&=&((I-z{\rm e}^{-i\frac t2 H_{u_0}^2}{\rm e}^{i t K_{u_0}^2}S^*{\rm e}^{-i\frac t2 H_{u_0}^2})^{-1}{\rm e}^{-i\frac t2 H_{u_0}^2}(u_0)\vert 
{\rm e}^{i\frac t2 H_{u_0}^2}(1))\\
&=& ((I-z{\rm e}^{-i t H_{u_0}^2}{\rm e}^{i t K_{u_0}^2}S^*)^{-1}{\rm e}^{-i t H_{u_0}^2}(u_0)\vert 1)\ ,
\end{eqnarray*}
which is the claimed formula in the case of data $u_0\in H^s_+, s>\frac 12$. The case $u_0\in H^{\frac 12}_+$ follows by a simple approximation argument. Indeed, we know from \cite{GG}, Theorem 2.1, that, for every $t\in \R $, the mapping $u_0\mapsto u(t)$ is continuous on $H^{\frac 12}_+$. On the other hand, the maps $u_0\mapsto H_{u_0}, K_{u_0}$ are continuous from $H^{\frac 12}_+$ into $\mathcal L(L^2_+)$. Since 
$H_{u_0}^2, K_{u_0}^2$ are selfadjoint, the operator 
$${\rm e}^{-i t H_{u_0}^2}{\rm e}^{i t K_{u_0}^2}S^*$$
has norm at most $1$. Hence, for $z\in D$, the right hand side of the formula is continuous from $H^{\frac 12}_+$ into $\C $.
\section{An example}\label{instab}
This section is devoted to revisiting sections 6.1, 6.2 of \cite{GG} by means of the explicit formula. Given $\e \in \R $, we define
$$u_0^\e (x)={\rm e}^{ix}+\e \ .$$
It is easy to check that $u_0^\e \in \mathcal V(3)$, hence the corresponding solution $u^\e $ of (\ref{szego}) is valued in 
$\mathcal V(3)$, and consequently reads
$$u^\e (t,x)=\frac{a^\e (t){\rm e}^{ix}+b^\e (t)}{1-p^\e (t){\rm e}^{ix}}\ ,$$
with $a^\e (t)\in \C^*, b^\e (t) \in \C , p^\e (t)\in D, a^\e (t)+b^\e (t)p^\e (t)\ne 0. $ We are going to calculate these functions explicitly.
We start with the special case $\e =0$. In this case, $\vert u_0^0 \vert =1$, hence 
$$u^0(t,x)={\rm e}^{-it}u^0 _0(x)$$
so 
$$a^0(t)={\rm e}^{-it}\ ,\ b^0(t)=0\ ,\ p^0(t)= 0\ .$$
We come to $\e \ne 0$. The operators $H_{u_0}^2, K_{u_0}^2, S^*$ act on the range of $H_{u_0^\e}$,
which  is the two dimensional vector space spanned by $1, {\rm e}^{ix}$. In this basis, the matrices of these three operators
are respectively
\begin{eqnarray*}
\mathcal M(H_{u_0}^2)= \left (\begin{array}{cc}  1+\e ^2 &\e \\ \e &1\end{array} \right ) \ ,\    \mathcal M(K_{u_0}^2)= \left (\begin{array}{cc}  1 &0\\ 0&0\end{array} \right )\ ,\ \mathcal M(S^*)= \left (\begin{array}{cc}  0 &1 \\ 0 &0\end{array} \right )\ .
\end{eqnarray*}
The eigenvalues of $H_{u_0}^2$ are
$$\rho _\pm ^2=1+\frac{\e ^2}2 \pm \e \sqrt{1+\frac{\e ^2}4}\ ,$$
hence the matrix of the exponential is given by
\begin{eqnarray*}\mathcal M\left ({\rm e}^{-it H_{u_0}^2}\right )&=&\frac{{\rm e}^{-it\rho _+^2}-{\rm e}^{-it\rho _-^2}}{\rho _+^2-\rho _-^2}\, \mathcal M(H_{u_0}^2)+\frac{\rho _-^2{\rm e}^{-it\rho _+^2}-
\rho _+^2{\rm e}^{-it\rho _-^2}}{\rho _-^2-\rho _+^2}\, I\ \\
&=& \frac{{\rm e}^{-i\Omega t}}{2\omega}\left ( -2i\sin \left (\omega t\right )\mathcal M(H_{u_0}^2)+(2\omega \cos (\omega t)+2i\Omega \sin (\omega t))I    \right )\\
\text{where }\ \omega &:=&\e \sqrt{1+\frac {\e ^2}4}\ ,\ \Omega :=1+\frac{\e ^2}2.
\end{eqnarray*}
We obtain
\begin{eqnarray*}
&&{\rm e}^{-it H_{u_0}^2}(u_0)= \frac{{\rm e}^{-i\Omega t}}{2\omega}\left (-2i\e \Omega \sin (\omega t)+2\e \omega \cos (\omega t)+(2\omega \cos (\omega t)-i\e ^2\sin (\omega t)){\rm e}^{ix}\right )\ ,\\
&&\mathcal M\left ( {\rm e}^{-i t H_{u_0}^2}{\rm e}^{i t K_{u_0}^2}S^* \right )=\frac{{\rm e}^{-i t\frac {\e ^2}{2}}}{2\omega}\left (\begin{array} {cc}  
0&  2\omega \cos (\omega t)-i\e ^2\sin (\omega t)   \\ 0& -2i\e \sin(\omega t)
\end{array}\right )\ ,
\end{eqnarray*}
and finally
\begin{eqnarray*}
a^\e (t)&=&{\rm e}^{-it(1+\e ^2)}\ ,\ b^\e (t)={\rm e}^{-it(1+\e ^2/2)}\left (\e \cos (\omega  t)
-i\frac{2+\e ^2}{\sqrt{4+\e ^2}}\sin (\omega  t)\right )\\
p^\e (t)&=&-\frac{2i}{\sqrt{4+\e ^2}}\sin (\omega  t)\, {\rm e}^{-it\e ^2/2}\ ,\ \omega  :=\frac \e 2\sqrt{4+\e ^2}\ .
\end{eqnarray*}
The important feature of such dynamics concerns the regime $\e \to 0$. Though $p^0(t)\equiv 0$, $p^\e (t)$ may visit small neighborhoods of the unit circle at large times.
\begin{figure}
\includegraphics[scale=0.25]{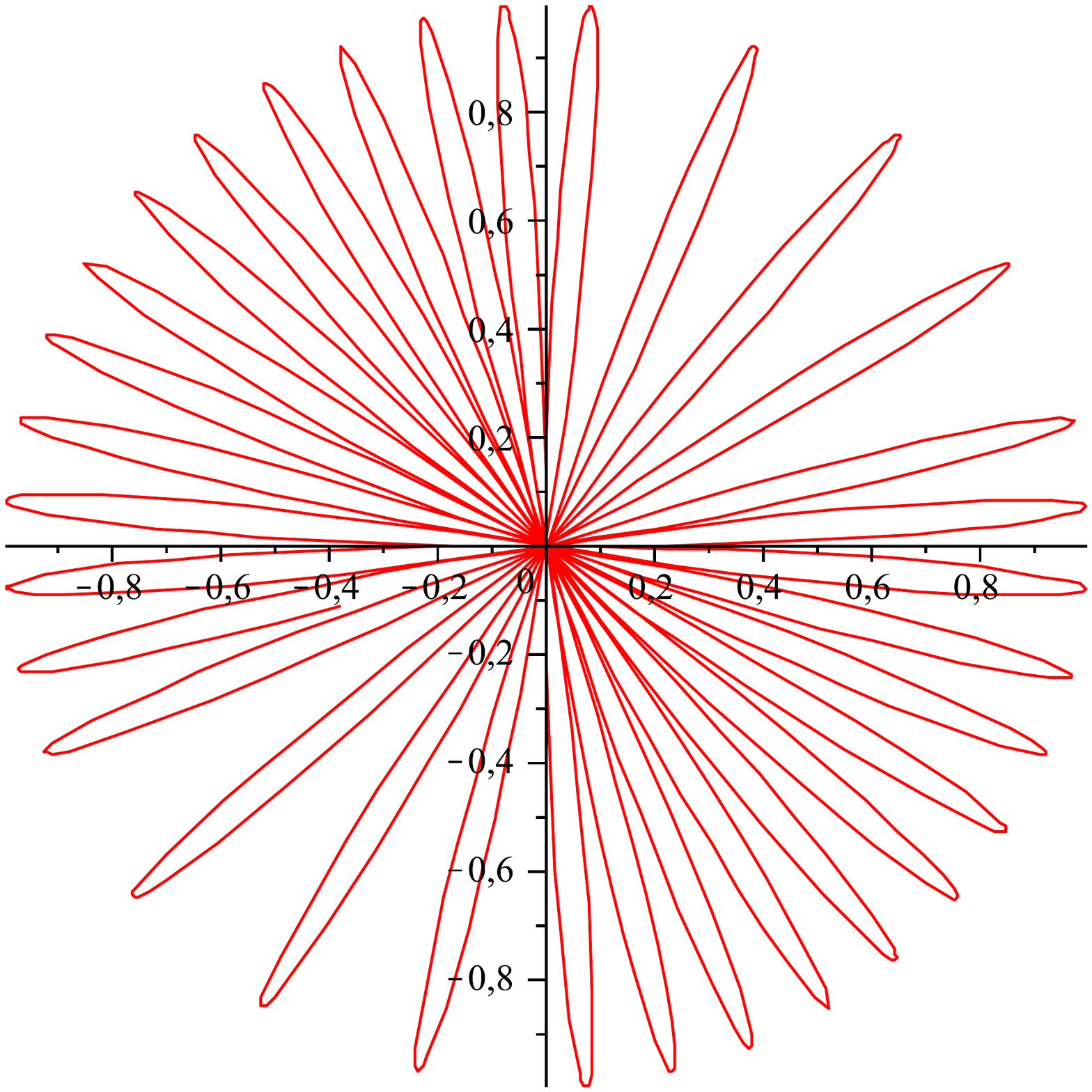}
\caption{The trajectory of $p^\e $ for  small $\e $.}
\end{figure}
Specifically, at time $t^\e =\pi /(2\omega )\sim \pi /(2\e )$, we have $\vert p^\e (t)\vert \sim 1-\e ^2$. A consequence is that the momentum density,
\begin{eqnarray*}
\mu _n(t^\e )&:=&n\vert \hat u^\e (t^\e ,n)\vert ^2=n\vert a^\e (t^\e )+b^\e (t^\e )p^\e (t^\e )\vert ^2\vert p^\e (t^\e )\vert ^{2(n-1)}\\
&=&n\frac{\e ^4}{(4+\e ^2)^2} \left (1-\frac {\e ^2}{4+\e ^2}\right )^{n-1},
\end{eqnarray*}
which satisfies 
$$\sum _{n=1}^\infty \mu _n(t^\e )= Tr(K_{u^\e (t^\e )}^2)= Tr(K_{u_0^\e }^2)=1\ ,$$
becomes concentrated at high frequencies
$$n\simeq \frac 1{\e ^2}\ . $$
This induces the following instability of $H^s$ norms
$$\Vert u^\e (t^\e )\Vert _{H^s}\simeq \frac 1{\e ^{2s-1}}\ ,\ s>\frac 12\ ,$$
a phenomenon of the same nature as the one displayed by Colliander, Keel, Staffilani, Takaoka and Tao in \cite{CKSTT}. This proves in particular that conservation laws do not control $H^s$ regularity for $s>\frac 12$. Notice that, as already mentioned at the end of subsection \ref{finite}
of the introduction, the family $(u_0^\e )$ approaches $u_0^0$, which is a non generic element of $\mathcal V(3)$, since $H_{u_0}^2$ admits $1$ as  a double eigenvalue.
\s
This example naturally leads  to the question of large time behavior of the $H^s$ norm of individual solutions for $s>\frac 12$. 
We are going to answer this question in the special case of finite rank solutions by proving the quasi periodicity theorem in the next two sections.
\section{Generalization to the Szeg\H{o} hierarchy}
The Szeg\H{o} hierarchy was introduced in \cite{GG} and used in \cite{GG2}. For the convenience of the reader, and because our notation is 
slightly different, we shall recall the main facts here. For $y>0$ and $u\in H^{\frac 12}_+$, we set
$$J^y(u)= ((I+yH_u^2)^{-1}(1)\vert 1)\ .$$
Notice that the connection with the Szeg\H{o} equation is made by
$$E(u)=\frac 14(\pa _y^2J^y_{\vert y=0}-(\pa _yJ^y_{\vert y=0})^2)\ .$$
For every $s>\frac 12$, $J^y $ is a smooth real valued function on $H^s_+$, and its Hamiltonian vector field  is given by
$$X_{J^y}(u)=2iy w^yH_uw^y\ ,\ w^y:=(I+yH_u^2)^{-1}(1)\ ,$$
which is a Lipschitz vector field on bounded subsets of $H^s_+$. This fact is a consequence of the following lemma, where we collect basic estimates. We recall that the Wiener algebra $W$ is  the space  of $f\in L^2_+$ such that
$$\Vert f\Vert _{W}:=\sum _{k=0}^\infty \vert \hat f(k)\vert <\infty \ .$$
\begin{lemma}\label{Lipschitz}
Let $f,u,v\in L^2_+$.
\begin{eqnarray*}
\Vert H_uf\Vert _{W}&\le &\Vert u\Vert _W\Vert f\Vert _W\ ,\\
\Vert H_uf\Vert _{H^{s-\frac 12}}&\le &\Vert u\Vert _{H^s}\Vert f\Vert _{L^2}\ ,\ s\ge \frac 12\ ,\\
\Vert H_uf\Vert _{H^s}&\le & \Vert u\Vert _{H^s}\Vert f\Vert _{W}\ ,\  s\ge 0\ ,\\
\Vert w^y \Vert _{H^s}&\le &(1+y\Vert u\Vert _{H^s}^2)\ ,\ s>1\ ,\\
\Vert fg\Vert _{H^s}&\le &C_s(\Vert f\Vert _W\Vert g\Vert _{H^s}+\Vert g\Vert _W\Vert f\Vert _{H^s})\ ,\\
\Vert X_{J^y}(u)-X_{J^y}(v)\Vert _{H^s}&\le &C_s(R,y)\Vert u-v\Vert _{H^s}\ ,\ s>1\ ,\Vert u\Vert _{H^s}+\Vert v\Vert _{H^s}\le R\ .
\end{eqnarray*}
\end{lemma}
\begin{proof}
The first three estimates are straightforward consequences of the formula
$$\widehat {H_uf}(k)=\sum _{\ell =0}^\infty \hat u(k+\ell )\overline {\hat f(\ell )}
\ .$$
The fourth estimate comes from these estimates and the fact that
$$w^y=1-yH_u^2w^y\ ,\ \Vert w^y\Vert _{L^2}\le 1\ .$$
The fifth estimate is obtained by decomposing
$$\widehat{fg}(k)=\sum _{\ell =0}^\infty \hat f(k-\ell )\hat g(\ell )=\sum _{\vert k-\ell \vert \le \ell } \hat f(k-\ell )\hat g(\ell )+
\sum _{\vert k-\ell \vert > \ell } \hat f(k-\ell )\hat g(\ell )\ .$$
As for the last estimate, we set
$$w^y[u]:=(I+yH_u^2)^{-1}(1)\ .$$
We  write
$$\Vert w^y[u]-w^y[v]\Vert _{L^2}=y\Vert (I+yH_u^2)^{-1}(H_v^2-H_u^2)(I+yH_v^2)^{-1}(1)\Vert _{L^2}\le yR\Vert u-v\Vert _{H^s}\ .$$
Then, by using again the first two  inequalities,
$$w^y[u]-w^y[v]=y(H_v^2(w^y[v])-H_u^2(w^y[u]))$$
leads to 
$$\Vert w^y[u]-w^y[v]\Vert _{H^s}\le C(R,y)\Vert u-v\Vert _{H^s}\ .$$
Using moreover the fact that $H^s$ is an algebra, this yields the desired estimate.
\end{proof}

By the Cauchy--Lipschitz theorem, the evolution equation 
\begin{equation}\label{hierarchy}
\dot u=X_{J^y}(u)
\end{equation}
admits local in time solutions for every initial data in $H^s_+$ for $s>1$, and the lifetime is bounded from below if the 
data are bounded in $H^s_+$. We shall see that this evolution equation admits a Lax pair structure similar to the 
one in section \ref{Lax}. 
\begin{theorem}\label{Laxhier}
For every $u\in H^s_+$, we have
\begin{eqnarray*}
H_{iX_{J^y}(u)}&=&H_u F_u^y+F_u^yH_u\ ,\\
K_{iX_{J^y}(u)}&=&K_uG_u^y+G_u^yK_u\ ,\\
 G_u^y(h)&:=&-yw^y\, \Pi (\overline {w^y}\, h)+y^2H_uw^y\, \Pi (\overline {H_uw^y}\, h)\ ,\\
 F_u^y(h)&:=&G_u^y(h)-y^2(h\vert H_uw^y)H_uw^y\ .
\end{eqnarray*}
If $u\in C^\infty (\mathcal I ,H^s_+)$ is a solution of equation (\ref{hierarchy}) on a time interval $\mathcal I$, then
\begin{eqnarray*}
\frac{dH_u}{dt}&=&[B_u^y,H_u]\ ,\ \frac{dK_u}{dt}=[C_u^y,K_u]\ ,\\
B_u^y&=&-iF_u^y\ ,\ C_u^y=-iG_u^y\ .
\end{eqnarray*}
\end{theorem}
\begin{proof}
\begin{lemma}\label{luminy}
We have the following identity,
$$H_{aH_u(a)}(h)=H_u(a)H_a(h)+H_u(a\Pi (\overline ah)-(h|a)a)\ .$$
\end{lemma}
\begin{proof}
$$H_{aH_u(a)}(h)=\Pi(aH_u(a)\overline h)=H_u(a)H_a(h)+\Pi (H_u(a)(I-\Pi)(a\overline h))\ .$$
On the other hand,
$$(1-\Pi)(a\overline h)=\overline{\Pi (\overline ah)}-(a|h)\ .$$
The lemma follows by plugging the latter formula into the former
one.
\end{proof}
Let us complete the proof. Using the identity
$$w^y=1-yH_u^2w^y,$$
and Lemma \ref{luminy} with $a=H_u(w^y)$, we get
\begin{eqnarray*}
&&H_{w^y\, H_u(w^y)}(h)=H_{H_u(w^y)}(h)-yH_{H_u(w^y)H_u^2(w^y)}(h)\\
&=&H_{H_u(w^y)}(h)-yH_u^2(w^y)H_{H_u(w^y)}(h)
- yH_u\left (H_u(w^y)\Pi(\overline {H_u(w^y)}h)-(h|H_u(w^y))H_u(w^y)\right )\\
&=&w^y\, H_{H_u(w^y)}(h)-yH_u\left (H_u(w^y)\Pi(\overline {H_u(w^y)}h)-(h|H_u(w^y))H_u(w^y)\right )\\
&=&w^y\, \Pi(\overline {w^y}\, H_uh)-yH_u\left (H_u(w^y)\Pi(\overline
{H_u(w^y)}h)-(h|H_u(w^y))H_u(w^y)\right )\ .
\end{eqnarray*}
We therefore have obtained
$$H_{w^y\, H_u(w^y)}=L_u^yH_u+H_uR_u^y$$
where $L_u^y$ and $R_u^y$ are the following self adjoint operators,
$$L_u^y(h)=w^y\, \Pi(\overline {w^y}\, h)\ ,\ R_u^y(h)=-y\left (H_u(w^y)\Pi(\overline {H_u(w^y)}h)-(h|H_u(w^y))H_u(w^y)\right )\ .$$
Consequently, since $H_{w^y\, H_u(w^y)}$ is self adjoint,
$$H_{w^y\, H_u(w^y)}=\frac 12 (L_u^y+R_u^y)H_u+H_u\frac 12 (L_u^y+R_u^y)\ .$$
Multiplying by $-2y$, we obtain the desired formula, since
$$F_u^y=-y(L_u^y+R_u^y)\ .$$
We now come to the second identity. From the first one, we get
\begin{equation}\label{KX}
K_{iX_{J^y}(u)}=H_{iX_{J^y}(u)}S=H_uF_u^yS+F_u^yK_u\ .
\end{equation}
For every $h, v\in L^2_+$, we use  $$\Pi (\overline vSh)=S\Pi (\overline vh)+(Sh\vert v)$$
and infer
\begin{eqnarray*}
F_u^ySh&=&-yw^y\, \Pi (\overline {w^y}\, Sh)+y^2H_uw^y\, \Pi (\overline {H_uw^y}\, Sh)-y^2(Sh\vert H_uw^y)H_uw^y\\
&=&SG_u^yh-y(Sh\vert w^y)w^y=SG_u^yh+y^2(Sh\vert H_u^2w^y)w^y\\
&=&SG_u^yh+y^2(H_u(w^y)\vert K_u(h))w^y\ ,
\end{eqnarray*}
where we have used  $w^y=1-yH_u^2w^y$ again. Plugging this identity into (\ref{KX}), we obtain the claim.
\s
The last formulae are straightforward consequences of the antilinearity of $H_u$ and $K_u$.
\end{proof}
Using Theorem \ref{Laxhier} in a similar way to section \ref{Lax}, we derive  
\begin{corollary}\label{UVy}
Under the conditions of Theorem \ref{Laxhier}, assuming moreover $0\in \mathcal I$, define $U^y=U^y(t)$, $V^y=V^y(t)$ the solutions of the following linear ODEs
on $\mathcal L(L^2_+)$,
$$\frac {dU^y}{dt}=B_u^y\, U^y\ ,\ \frac{dV^y}{dt}=C_u^y\, V^y\ ,\ U^y(0)=V^y(0)=I\ .$$
Then $U^y(t), V^y(t)$ are unitary operators and
$$H_{u(t)}=U^y(t)H_{u(0)} U^y(t)^*\ ,\ K_{u(t)}=V^y(t)K_{u(0)} V^y(t)^*\ .$$
\end{corollary}
At this stage, we are going to generalize slightly the setting, for the needs of the next section. Let $y_1,\dots ,y_n$ be positive numbers,
and $a_1,\dots ,a_n$ be real numbers. We consider the functional
$$\hat J(u)=\sum _{k=1}^na_kJ^{y_k}(u)=(f(H_u^2)1\vert 1)\ ,\ f(s):=\sum _{k=1}^n \frac{a_k}{1+y_ks}\ ,$$
and the evolution equation
\begin{equation}\label{evolhat}
\dot u=X_{\hat J}(u)\ .
\end{equation}
By linearity from Theorem \ref{Laxhier}, it is clear that the solution of (\ref{evolhat}) satisfies
\begin{equation}\label{Laxhat}
\frac{dH_u}{dt}=[\hat B_u,H_u]\ ,\ \frac{dK_u}{dt}=[\hat C_u,K_u]\ ,
\end{equation}
with
\begin{equation}\label{BChat}
\hat B_u=\sum _{k=1}^na_k B_u^{y_k}\ ,\ \hat C_u=\sum _{k=1}^n a_kC_u^{y_k}\ .
\end{equation}
\begin{corollary}\label{UVhat}
Let $u$ be a solution of equation (\ref{evolhat}) on some time interval $\mathcal I$ containing $0$, define $\hat U=\hat U(t)$, $\hat V=\hat V(t)$ the solutions of the following linear ODEs
on $\mathcal L(L^2_+)$,
$$\frac {d\hat U}{dt}=\hat B_u\, \hat U\ ,\ \frac{d\hat V}{dt}=\hat C_u\, \hat V\ ,\ \hat U(0)=\hat V(0)=I\ .$$
Then $\hat U(t), \hat V(t)$ are unitary operators and
$$H_{u(t)}=\hat U(t)H_{u(0)} \hat U(t)^*\ ,\ K_{u(t)}=\hat V(t)K_{u(0)} \hat V(t)^*\ .$$
\end{corollary}
As a consequence of this corollary, if we start from an initial datum $u(0)$ such that $H_{u(0)}$ is a trace class operator, 
then $H_{u(t)}$ is trace class for every $t$, with the same trace norm. By Peller's theorem \cite{P}, Chap. 6, Theorem 1.1, 
 the trace norm of $H_u$ is equivalent to 
the norm of $u$ in the Besov space $B^1_{1,1}$, which is contained into $W$ and contains $H^s_+$ for every $s>1$. Consequently, if 
$u(0)\in H^s_+$ for some $s>1$, then $u(t)$ stays bounded in $W$. We claim that, if $u(0)$ is in $\mathcal V(d)$, the evolution can be continued for all time. Moreover,
since the ranks of $H_{u(t)}$ and $K_{u(t)}$ are conserved in view of Corollary \ref{UVhat}, this evolution takes place in $\mathcal V(d)$
if $u(0)\in \mathcal V(d)$.
\begin{corollary}\label{globalflow}
The equation (\ref{evolhat}) defines a smooth flow on $H^s_+$ for every $s>1$ and on $\mathcal V(d)$ for every $d$.
\end{corollary}
In view of the Gronwall lemma, the statement is an easy consequence of the following estimate.
\begin{lemma}\label{4avril}
Let $R, y\ge 0, s>1$ be given. There exists $C(d,R,y,s)>0$ such that, for every $u\in \mathcal V(d)$ with $\Vert u\Vert _W\le R$, 
$$\Vert X_{J^y}(u)\Vert _{H^s}\le C(d,R,y,s)(1+\Vert u\Vert _{H^s})\ .$$
\end{lemma}
\begin{proof} By using Lemma \ref{Lipschitz},
we are reduced to prove
$$\Vert w^y\Vert _{W}\le B(d,R,y)\ .$$
We set $N=\left [ \frac {d+1}2\right ]\ .$ The above estimate is an easy consequence of
$$(I+H_u^2)^{-1}=\sum _{k=0}^{N}a_kH_u^{2k}\ ,$$
with $\vert a_k\vert \le 1$ for $k=0,\dots ,N$. In fact,  the Cayley--Hamilton theorem yields
$$(H_u^2)^{N+1}=\sum _{k=1}^N (-1)^{k-1}S_k(H_u^2)^{N-k+1}\ ,\ S_k:=\sum _{\ell _1<\dots <\ell _k}\rho _{\ell _1}^2\dots \rho_{\ell _k}^2\ ,$$
and  one can  easily check that
$$a_k=(-1)^k\ \frac{\displaystyle{1+\sum _{j=1}^{N-k}S_j}}{\displaystyle {1+\sum _{j=1}^NS_j}}\ ,\ k=0,\dots ,N\ .$$
where $\rho _1^2\ge \dots \ge \rho _N^2$ are the positive eigenvalues of $H_u^2$, listed with their multiplicities.
\end{proof}
\begin{remark}
For general data $u(0)\in H^s_+$, one can prove similarly that the solution can be continued for all time if
$y\Vert u(0)\Vert _{H^s}$ is small enough, or just if $y\, {\rm Tr}\vert H_{u(0)}\vert $ is small enough.
\end{remark}
Our next step is to derive an explicit formula for the solution of (\ref{evolhat}) along the same lines as in section \ref{preuve}.
The starting points are the formulae
\begin{eqnarray*}
B_u^y(1)&=&iyJ^y(u)w^y\\
C_u^y-B_u^y&=&-iy^2(\cdot \vert H_uw^y)H_uw^y\\
&=&iyJ^y(u)((I+yH_u^2)^{-1}-(I+yK_u^2)^{-1})\ ,
\end{eqnarray*}
where we have used the identity $K_u^2=H_u^2-(\cdot \vert u)u\ .$ This leads to
\begin{eqnarray*}
\hat B_u(1)&=&ig(H_u^2)(1)\ ,\ g(s):=\sum _{k=1}^n \frac{a_ky_kJ^{y_k}(u)}{1+y_ks}\ ,\\
\hat C_u -\hat B_u&=&i(g(H_u^2)-g(K_u^2))\ .
\end{eqnarray*}
Arguing exactly as in section \ref{preuve}, we obtain the following formula.
\begin{theorem}\label{explicithat}
 The solution $u$ of equation
(\ref{evolhat}) with initial data $u(0)=u_0\in H^s_+, s>1,$ is given by
\begin{equation}
\underline u(t,z)=((I-z{\rm e}^{2itg(H_{u_0}^2)}{\rm e}^{-2itg(K_{u_0}^2)}S^*)^{-1}{\rm e}^{2itg(H_{u_0}^2)}u_0\, \vert \, 1)\ ,\ z\in D\ ,
\end{equation}
where
$$g(s):=\sum _{k=1}^n \frac{a_ky_kJ^{y_k}(u)}{1+y_ks}\ .$$
\end{theorem}
\section{Proof of the quasiperiodicity theorem}
In this section, we prove Theorem \ref{quasip}. Let $u_0\in \mathcal V(d)$ be given. Firstly we show that $t\mapsto u(t)$ is a quasi periodic 
function valued into $\mathcal V(d)$. Denote by $\Sigma $ the union of the spectra of 
$H_{u_0}^2$ and $K_{u_0}^2$. We claim that it is enough to prove that, for any function $\omega  :\Sigma \to \T $, the formula
$$\underline {\Phi (\omega )}(z)= ((I-z{\rm e}^{-i\omega }(H_{u_0}^2){\rm e}^{i\omega }(K_{u_0}^2)S^*)^{-1}{\rm e}^{-i\omega }(H_{u_0}^2)u_0\, \vert \, 1)\ ,\ z\in D\ ,$$
defines an element $\Phi (\omega ) \in \mathcal V(d)$. Indeed, if this is 
established, Theorem \ref{explicit} exactly claims that $u(t)=\Phi (t\omega )$, where, for every $s\in \Sigma $, $\omega (s)=s \text{ mod } 2\pi\ .$
Moreover, it is clear from the above formula that $\Phi (\omega )$ is a rational function with coefficients smoothly dependent on $\omega \in \T ^\Sigma $, so that $\Phi $ is smooth as a map from $\T ^\Sigma $ to $\mathcal V(d)$.
\s
Let $\omega \in \T ^\Sigma $. For each $s\in \Sigma $, we represent $\omega (s)$ by some element of $[0,2\pi )$, still denoted
by $\omega (s)$. Fix $n=\vert \Sigma \vert $ and let $y_1,\dots ,y_n$ be $n$ positive numbers pairwise distinct. Then the matrix
$$\left (\frac 1{1+y_ks}\right )_{k=1,\dots ,n, s\in \Sigma }$$
is invertible, hence the linear system
$$\omega (s)=-2\sum _{k=1}^n \frac{a_ky_kJ^{y_k}(u_0)}{1+y_ks}\ ,\ s\in \Sigma \ $$
has a unique solution $a_1,\dots ,a_n$. Using Theorem \ref{explicithat}, $\Phi (\omega )$ is the value at time $t=1$
of the solution $u$ of equation (\ref{evolhat}) with parameters $a_1,\dots ,a_n,y_1,\dots ,y_n$. By Corollary \ref{globalflow}, it belongs
to $\mathcal V(d)$. This proves quasi periodicity.
\s
Since $\Phi $ is a continuous mapping, $\Phi (\T ^\Sigma )$ is a compact subset of $\mathcal V(d)$. On the other hand, for every
$s$, the $H^s$ norm is continuous on $\mathcal V(d)$. It is therefore bounded on this compact subset, which contains the integral curve
issued from $u_0$. This completes the proof of Theorem \ref{quasip}.
\begin{remark} It is tempting to adapt the above proof of quasi periodicity to non finite rank solutions. However, even assuming that one can define a flow on $H^s_+$ for all $y$ with convenient estimates for large $y$, this strategy meets 
a serious difficulty. Indeed, on the one hand, the construction of a Hamiltonian flow on $H^s_+$ for 
$$\hat J(u)=(f(H_u^2)1\vert 1)$$
requires a minimal regularity for $f$, say $C^1$, which, if $f$ is represented as
$$f(s)=\int _0^\infty \frac{a(y)}{1+ys}\, d\mu (y)$$
for some positive measure $\mu $  and some function $a$ on $\R _+$, imposes a decay condition as
$$\int _0^\infty y\vert a(y)\vert \, d\mu (y)\ .$$
On the other hand, $\Sigma $ is made of  a sequence of positive numbers converging to $0$ and of its limit, and the interpolation problem
$$\omega (s)=-2\int _0^\infty \frac{ya(y)J^y(u_0)}{1+ys}\, d\mu (y)$$
would have a solution only if $\omega :\Sigma \to \T $ is continuous on $\Sigma $. Unfortunately,
the space $C(\Sigma , \T )$ is not compact, neither for the simple convergence, nor for the uniform convergence.
Therefore the question of large time dynamics of non finite rank solutions of the cubic Szeg\H{o} equation remains widely open. 
\end{remark}

\end{document}